%% file: invariant_form.tex
\documentclass{amsart}

\theoremstyle{definition}

\input{header}

\begin{document}

\title[Some interesting conformally invariant one-forms]{An interesting family of conformally invariant one-forms in even dimensions}
\author{Jeffrey S. Case}
\thanks{JSC was supported by a grant from the Simons Foundation (Grant No.\ 524601)}
\address{109 McAllister Building \\ Penn State University \\ University Park, PA 16802}
\email{jscase@psu.edu}
\begin{abstract}
 We construct a natural conformally invariant one-form of weight $-2k$ on any $2k$-dimensional pseudo-Riemannian manifold which is closely related to the Pfaffian of the Weyl tensor.  On oriented manifolds, we also construct natural conformally invariant one-forms of weight $-4k$ on any $4k$-dimensional pseudo-Riemannian manifold which are closely related to top degree Pontrjagin forms.  The weight of these forms implies that they define functionals on the space of conformal Killing fields.  On Riemannian manifolds, we show that this functional is trivial for the former form but not for the latter forms.  As a consequence, we obtain global obstructions to the existence of an Einstein metric in a given conformal class.
\end{abstract}
\maketitle

\input{intro}
\input{bg}
\input{conf}
\input{proofs}
\input{pontrjagin}
\input{natural}

\subsection*{Acknowledgments}
I thank the anonymous referee for helpful comments which led to an improved version of \cref{prop:unnatural}.

\bibliography{bib}
\end{document}

%% file: header.tex
\usepackage{amsmath}
\usepackage{amssymb}
\usepackage{amsthm}
\usepackage[msc-links,abbrev]{amsrefs}
\usepackage{hyperref}

\usepackage[noabbrev,capitalize]{cleveref}

\DeclareMathOperator{\Id}{Id}

\DeclareMathOperator{\im}{im}

\DeclareMathOperator{\tr}{tr}

\DeclareMathOperator{\dvol}{dvol}

\DeclareMathOperator{\Ric}{Ric}

\DeclareMathOperator{\Rm}{Rm}

\DeclareMathOperator{\End}{End}

\DeclareMathOperator{\Met}{Met}

\DeclareMathOperator{\Pf}{Pf}

\DeclareMathOperator{\tf}{tf}

\newcommand{\og}{\overline{g}}

\newcommand{\cPhi}{\widetilde{\Phi}}

\newcommand{\hg}{\widehat{g}}

\newcommand{\hT}{\widehat{T}}
\newcommand{\hW}{\widehat{W}}

\newcommand{\hnabla}{\widehat{\nabla}}

\newcommand{\hrho}{\widehat{\rho}}

\newcommand{\hxi}{\widehat{\xi}}

\newcommand{\lp}{\langle}
\newcommand{\rp}{\rangle}
\newcommand{\lv}{\lvert}
\newcommand{\rv}{\rvert}

\newcommand{\bCP}{\mathbb{C}P}

\newcommand{\mE}{\mathcal{E}}

\newcommand{\mK}{\mathcal{K}}
\newcommand{\mL}{\mathcal{L}}

\newcommand{\bH}{\mathbb{H}}

\newcommand{\bN}{\mathbb{N}}

\newcommand{\bR}{\mathbb{R}}

\newcommand{\Rho}{\mathrm{P}}

\def\sideremark#1{\ifvmode\leavevmode\fi\vadjust{\vbox to0pt{\vss
 \hbox to 0pt{\hskip\hsize\hskip1em
 \vbox{\hsize3cm\tiny\raggedright\pretolerance10000
 \noindent #1\hfill}\hss}\vbox to8pt{\vfil}\vss}}}

\newcommand{\comment}[1]{}

\newtheorem{theorem}{Theorem}[section]
\newtheorem{lemma}[theorem]{Lemma}
\newtheorem{proposition}[theorem]{Proposition}
\newtheorem{corollary}[theorem]{Corollary}

\theoremstyle{definition}

\theoremstyle{remark}
\newtheorem{remark}[theorem]{Remark}

\numberwithin{equation}{section}

%% file: intro.tex
\section{Introduction}
\label{sec:intro}

Recent work~\cites{CaseGover2013,CaseTakeuchi2019,Marugame2019} in CR geometry has identified an interesting family of natural CR invariant $(1,0)$-forms on all nondegenerate CR manifolds of dimension $2n+1$, $n\geq2$.  These $(1,0)$-forms can be regarded as CR invariant modifications of $\partial_b c_\Phi(S)$, where $c_\Phi(S)$ is the potential of a characteristic form of degree $2n$ determined by a homogeneous invariant polynomial $\Phi$ and the Chern tensor $S$.  For strictly pseudoconvex CR manifolds, a result of Takeuchi~\cite{Takeuchi2018} implies that these $(1,0)$-forms are all divergences.  This fact leads to counterexamples to Hirachi's conjecture on the generalization of the Deser--Schwimmer conjecture to CR geometry~\cite{Hirachi2013}.

The purpose of this article is to construct the conformal analogues of the above CR invariant one-forms.  The forms we construct retain three key properties of their CR analogues.  First, they are \emph{natural}; that is, they can be written as a linear combination of partial contractions of tensor products of the pseudo-Riemannian metric, its inverse, the Riemann curvature tensor, and its covariant derivatives; when restricted to oriented manifolds, we also allow these products to include factors of the pseudo-Riemannian volume form.  Second, they can be regarded as conformally invariant modifications of the exterior derivative of the Pfaffian of the Weyl tensor or, in the oriented case, the potential of a top degree Pontrjagin form.  Third, a result of Ferrand~\cite{Ferrand1971} and Obata~\cite{Obata1971} implies that, in Riemannian signature, the conformally invariant one-form related to the Pfaffian of the Weyl tensor is a divergence.  The conformally invariant one-forms related to top degree Pontrjagin forms need not be divergences, and their failure to be a divergence obstructs the existence of an Einstein metric in a given conformal class.

To make these points explicit requires some notation.  Let $(M^n,g)$ be a pseudo-Riemannian manifold.  Let $W_{ijpq}$ and $C_{ijp}$ denote the Weyl and Cotton tensors, respectively, with the convention $\nabla^pW_{ijpq}=(n-3)C_{ijq}$; here and throughout we use Penrose's abstract index notation~\cite{Penrose1984}.  Given $k\in\bN$, define
\begin{equation}
 \label{eqn:xi}
 \xi_i^{(k)} := \frac{1}{k!}\delta_{ii_2\dotsm i_{2k}}^{jj_2\dotsm j_{2k}} C_{jj_2}{}^{i_2}W_{j_3j_4}{}^{i_3i_4}\dotsm W_{j_{2k-1}j_{2k}}{}^{i_{2k-1}i_{2k}} + \frac{1}{2nk}\nabla_i \Pf^{(k)}(W),
\end{equation}
where $\delta_{ii_2\dotsm i_{2k}}^{jj_2\dotsm j_{2k}}$ is the generalized Kronecker delta and
\begin{equation}
 \label{eqn:Pf}
 \Pf^{(k)}(W) := \frac{1}{k!}\delta_{i_1\dotsm i_{2k}}^{j_1\dotsm j_{2k}} W_{j_1j_2}{}^{i_1i_2} \dotsm W_{j_{2k-1}j_{2k-2}}{}^{i_{2k-1}i_{2k}} .
\end{equation}
In dimension $n=2k$, it holds that $\Pf^{(k)}(W)$ is the Pfaffian of the Weyl tensor.

Suppose additionally that $(M^n,g)$ is an even-dimensional oriented manifold.  Set $n=2k$.  Denote by $\epsilon_{i_1\dotsm i_n}$ the pseudo-Riemannian volume form.  Let $\Phi$ be a homogeneous invariant polynomial of degree $k$; i.e.\ $\Phi$ is a linear combination of compositions of $\Id^{\otimes k}$ with braiding maps such that
\[ \Phi_{i_1\dotsm i_k}^{j_1\dotsm j_k} = \Phi_{i_{\sigma(1)}\dotsm i_{\sigma(k)}}^{j_{\sigma(1)}\dotsm j_{\sigma(k)}} \]
for all elements $\sigma$ of $S_k$, the symmetric group on $k$ elements.  Define
\begin{multline}
 \label{eqn:rho-Phi}
 \rho_i^\Phi := \frac{1}{(2k-1)!}\epsilon_i{}^{i_2\dotsm i_{2k}}\Phi_{s_1\dotsm s_k}^{t_1\dotsm t_k}C_{t_1}{}^{s_1}{}_{i_2}W_{t_2}{}^{s_2}{}_{i_3i_4}\dotsm W_{t_k}{}^{s_k}{}_{i_{2k-1}i_{2k}} \\ + \frac{1}{2k}\nabla_i p_\Phi(W),
\end{multline}
where
\begin{equation}
 \label{eqn:p-Phi} p_\Phi(W) := \frac{1}{(2k)!}\varepsilon^{i_1\dotsm i_{2k}} \Phi_{s_1\dotsm s_k}^{t_1\dotsm t_k} W_{t_1}{}^{s_1}{}_{i_1i_2} \dotsm W_{t_k}{}^{s_k}{}_{i_{2k-1}i_{2k}} .
\end{equation}
Note that $p_\Phi(W)=0$ if $k$ is odd and that $p_\Phi(W)=p_\Phi(\Rm)$ for all $k\in\bN$, where $p_\Phi(\Rm)$ is defined in terms of the Riemann curvature tensor $R_{ijkl}$ using \cref{eqn:p-Phi}.  The latter observation recovers the well-known fact~\cites{ChernSimons1974,BransonGover2007} that the Pontrjagin form $\star p_\Phi(\Rm)$ determined by $\Phi$ depends only on the Weyl tensor of $(M^n,g)$.

The one-form $\xi_i^{(k)}$ is conformally invariant in dimension $n=2k$ and the one-forms $\rho_i^\Phi$ are conformally invariant in the dimensions where they are defined.

\begin{theorem}
 \label{thm:invariant}
 Let $(M^{2k},g)$ be a pseudo-Riemannian manifold and let $\Phi$ be a homogeneous invariant polynomial of degree $2k$.  Then
 \begin{align*}
  e^{2k\Upsilon}\hxi_i^{(k)} & = \xi_i^{(k)} , \\
  e^{2k\Upsilon}\hrho_i^{\Phi} & = \rho_i^\Phi 
 \end{align*}
 for all conformal metrics $\hg:=e^{2\Upsilon}g$, where $\hxi_i^{(k)}$ and $\hrho_i^\Phi$ are defined in terms of $\hg$.
\end{theorem}

In terms of conformal density bundles, \cref{thm:invariant} states that $\xi_i^{(k)}$ and $\rho_i^\Phi$ are natural conformally invariant elements of $\mE_i[-2k]$ in dimension $2k$; see \cref{sec:bg} for definitions.  In particular, $\xi_i^{(k)}$ defines a conformally invariant functional on the space of compactly-supported vector fields.  More generally, let $(M^n,g)$ be a pseudo-Riemannian manifold.  Given an element $\omega_i\in\mE_i[-n]$, the formula
\begin{equation}
 \label{eqn:Omega}
 \Omega(X^i) := \int_M \omega_iX^i\,\dvol
\end{equation}
defines a conformally invariant functional on the space of compactly-supported vector fields on $M$.  More significantly, $\mE_i[-n]$ is the codomain of the formal adjoint $K^\ast\colon\mE_{(ij)_0}[2-n]\to\mE_i[-n]$ of the conformal Killing operator $K\colon\mE_i[2]\to\mE_{(ij)_0}[2]$, where $\mE_{(ij)_0}[w]$ denotes the space of conformally invariant, trace-free symmetric $(0,2)$-tensor fields with weight $w\in\bR$.   These operators are both conformally invariant, and the operator $K^\ast$ is a divergence: $K^\ast(T_{ij}):=-2\nabla^jT_{ij}$.

It is thus natural to ask whether $\xi_i^{(k)}$ or $\rho_i^\Phi$ are in the image of $K^\ast$.  A necessary condition is that, on compact manifolds, the associated functional $\Xi^{(k)}$ or $\Rho^\Phi$ annihilates conformal Killing fields.  For Riemannian manifolds, the fact that $K^\ast$ has surjective principal symbol implies that this condition is also sufficient.

On closed Riemannian manifolds, $\xi_i^{(k)}$ is in the image of $K^\ast$.

\begin{theorem}
 \label{thm:image}
 Let $(M^{2k},g)$ be a closed Riemannian manifold.  Then
 \[ \xi_i^{(k)} \in \im \left( K^\ast \colon \mE_{(ij)_0}[2-2k] \to \mE_i[-2k] \right) . \]
\end{theorem}

This result is remarkable due to the fact that $\xi_i^{(2)}$ is not the divergence of a natural trace-free symmetric $(0,2)$-tensor field; see \cref{sec:natural}.
To the best of our knowledge, this is the first example of a natural conformally invariant tensor field which is in the image of a natural conformally invariant differential operator, but is not the image of a natural tensor field.
By contrast, in dimension four, the Bach tensor
\begin{equation*}
 B_{ij} := \nabla^sC_{sij} + W_{isjt}P^{st} \in \mE_{(ij)_0}[-2]
\end{equation*}
is the image of the Weyl tensor under the natural conformally invariant differential operator $W_{ijkl} \mapsto (\nabla^s\nabla^t + P^{st})W_{isjt}$ (cf.\ \cite{GoverPeterson2005}).

Our proof of \cref{thm:image} relies on the Ferrand--Obata Theorem~\cites{Ferrand1971,Obata1971}.  Taken together, \cref{thm:invariant,thm:image} indicate that $\xi_i^{(k)}$ should be regarded as the conformal analogue of the aforementioned CR invariant $(1,0)$-forms.

By contrast, the one-forms $\rho_i^\Phi$ need not be in the image of $K^\ast$.  In fact, the failure of this to hold gives a global obstruction to the existence of an Einstein metric in the given conformal class.

\begin{theorem}
 \label{thm:pontrjagin-not-image}
 Let $\Phi$ be a homogeneous invariant polynomial of degree $2k$, $k\in\bN$.
 \begin{enumerate}
  \item If $(M^{4k},g)$ is a closed conformally Einstein manifold of Riemannian signature, then $\rho_i^\Phi\in\im K^\ast$.
  \item There are examples of closed manifolds $(M^{4k},g)$ for which $\rho_i^\Phi\not\in\im K^\ast$.
 \end{enumerate}
\end{theorem}

The proof of the first statement relies on the fact that, except on the round sphere, any conformal Killing field on a closed Einstein manifold of Riemannian signature is necessarily Killing~\cite{Obata1962}.  In \cref{sec:pontrjagin}, we show that the product of $S^1$ and a non-round Berger three-sphere, as well as its products with copies of $\bCP^2$, give examples with $\rho_i^\Phi\not\in\im K^\ast$.
Our examples are not locally conformally Einstein.
We are not aware of an example of a locally conformally Einstein manifold which can be proven via \cref{thm:pontrjagin-not-image} to not be globally conformally Einstein.

Note that on locally conformally flat and obstruction flat even-dimensional $n$-manifolds, $K^\ast\colon\mE_{(ij)_0}\to\mE_i[-n]$ is the last nontrivial map in the conformal deformation complex~\cites{GoverPeterson2005,GasquiGoldschmidt1984} and the conformal deformation detour complex~\cite{BransonGover2007b}, respectively.  In particular, \cref{thm:image,thm:pontrjagin-not-image} indicate that there may be an interesting interpretation of the conformally invariant one-forms $\xi_i^{(k)}$ and $\rho_i^\Phi$ on even-dimensional obstruction flat manifolds.

As previously noted, $\xi_i^{(k)}$ is not the divergence of a natural trace-free symmetric $(0,2)$-tensor field.  However, one can express $\xi_i^{(k)}$ as the sum of the divergence of a natural trace-free symmetric $(0,2)$-tensor field and the exterior derivative of a natural scalar function.

\begin{theorem}
 \label{thm:pfaffian}
 Let $(M^{2k},g)$ be a pseudo-Riemannian manifold.  Define $\Omega_{ij}^{(k)}\in\Gamma(S^2T^\ast M)$ by
 \[ \bigl(\Omega^{(k)}\bigr)_i^j := \sum_{\ell=0}^{k-1} 4^{k-\ell}\frac{1}{\ell!(k-\ell)}\delta_{ii_1\dotsm i_{k+\ell}}^{jj_1\dotsm j_{k+\ell}} W_{j_1j_2}{}^{i_1i_2}\dotsm W_{j_{2\ell-1}j_{2\ell}}{}^{i_{2\ell-1}i_{2\ell}}P_{j_{2\ell+1}}^{i_{2\ell+1}}\dotsm P_{j_{k+\ell}}^{i_{k+\ell}} , \]
 where $P_{ij}$ is the Schouten tensor of $g$.  Then
 \begin{equation}
  \label{eqn:pfaffian}
  2k\xi_i^{(k)} = \nabla^j(\tf\Omega^{(k)})_{ij} + \frac{1}{2k}\nabla_i\Pf^{(k)}(\Rm) ,
 \end{equation}
 where $(\tf\Omega^{(k)})_{ij} := \Omega_{ij}^{(k)} - \frac{1}{2k}\tr\Omega^{(k)}\,g_{ij}$ is the trace-free part of $\Omega_{ij}^{(k)}$.
\end{theorem}

There is a nice heuristic based on Branson's method of analytic continuation in the dimension~\cite{Branson1995} which explains \cref{thm:invariant,thm:image,thm:pfaffian}.  Let $(M^n,g)$ be a pseudo-Riemannian manifold and fix $k\in\bN$.  Define
\[ T^{(k)}(W)_i^j := \frac{1}{k!}\delta_{ii_1\dotsm i_{2k}}^{j j_1\dotsm j_{2k}} W_{j_1j_2}{}^{i_1i_2}\dotsm W_{j_{2k-1}j_{2k}}{}^{i_{2k-1}i_{2k}} . \]
Observe that $T^{(k)}(W)_{ij}$ is conformally invariant and $T^{(k)}(W)_{ij}=0$ if $n\leq 2k$.  Straightforward computations establish that
\begin{equation}
 \label{eqn:ac_divergence}
 \nabla^j\bigl(\tf T^{(k)}(W)\bigr)_{ij} = -2k(n-2k)\xi_i^{(k)}
\end{equation}
and
\begin{equation}
 \label{eqn:ac_conformal}
 e^{2k\Upsilon}\hnabla^j\bigl(\tf \hT^{(k)}(\hW)\bigr)_{ij} = \nabla^j\bigl(\tf T^{(k)}(W)\bigr)_{ij} + (n-2k)\Upsilon^i\bigl(\tf T^{(k)}(W)\bigr)_{ij}
\end{equation}
for all $\hg:=e^{2\Upsilon}g$.  Combining~\cref{eqn:ac_divergence,eqn:ac_conformal} yields
\[ e^{2k\Upsilon}\hxi_i^{(k)} = \xi_i^{(k)} - \frac{1}{2k}\Upsilon^i\bigl(\tf T^{(k)}(W)\bigr)_{ij} \]
when $n>2k$.  \Cref{thm:invariant} follows by taking the limit $n\to2k$.  \Cref{eqn:ac_divergence} exhibits $\xi_i^{(k)}$ in the image of the divergence on $\mE_{(ij)_0}$; dividing by $n-2k$ and taking the limit $n\to 2k$ yields \cref{thm:image}, provided one can make sense of the limit
\begin{equation}
 \label{eqn:problem_limit}
 \lim_{n\to 2k} \frac{1}{n-2k}\left(\tf T^{(k)}(W)\right)_{ij} .
\end{equation}
Finally, the generalized Einstein tensor
\[ \bigl(E^{(k)}\bigr)_i^j := \frac{1}{k!}\delta_{ii_1\dotsm i_{2k}}^{jj_1\dotsm j_{2k}} R_{j_1j_2}{}^{i_1i_2}\dotsm R_{j_{2k-1}j_{2k}}{}^{i_{2k-1}i_{2k}} \]
is symmetric and divergence-free~\cite{Lovelock1971}.  Note that $\tr E^{(k)}=(n-2k)\Pf^{(k)}(\Rm)$ and
\[ E_{ij}^{(k)} = T^{(k)}(W)_{ij} + (n-2k)\Omega_{ij}^{(k)}, \]
where
\begin{multline*}
 \bigl(\Omega^{(k)}\bigr)_i^j := \sum_{\ell=0}^{k-1} 4^{k-\ell}\binom{k}{\ell}\frac{(n-k-\ell-1)!}{k!(n-2k)!}\delta_{ii_1\dotsm i_{k+\ell}}^{jj_1\dotsm j_{k+\ell}} \\ \times W_{j_1j_2}{}^{i_1i_2}\dotsm W_{j_{2\ell-1}j_{2\ell}}{}^{i_{2\ell-1}i_{2\ell}}P_{j_{2\ell+1}}^{i_{2\ell+1}}\dotsm P_{j_{k+\ell}}^{i_{k+\ell}} .
\end{multline*}
In particular,
\[ -\frac{n-2k}{n}\nabla_i\Pf^{(k)}(\Rm)  = \nabla^j\bigl(\tf T^{(k)}(W)\bigr)_{ij} + (n-2k)\nabla^j\bigl(\tf\Omega^{(k)}\bigr)_{ij} . \]
Combining this with \cref{eqn:ac_divergence}, dividing by $n-2k$, and taking the limit $n\to 2k$ yields \cref{thm:pfaffian}.

We do not here attempt to make rigorous sense of the limit $n\to2k$.  Indeed, the failure of $\xi_i^{(2)}$ to be the divergence of a natural element of $\mE_{(ij)_0}[-2]$ in dimension four indicates that it is particularly difficult to make sense of \cref{eqn:problem_limit}.  Instead, we give direct proofs of \cref{thm:invariant,thm:pfaffian} using elementary multilinear algebra and then deduce \cref{thm:image} from \cref{thm:pfaffian} and the Ferrand--Obata Theorem.

The above heuristic also illustrates the distinction between the one-forms $\xi_i^{(k)}$ and $\rho_i^\Phi$, namely through how they are naturally extended to other dimensions.  In terms of the wedge product and Hodge star on double forms~\cite{Labbi2005}, the discussion above realizes $\xi_i^{(k)}$ as the divergence of a dimensional multiple of $\star (W^{\wedge k} \wedge g^{\wedge(n-2k-1)})$ when $n>2k$.  By contrast, the natural extension of $\rho_i^\Phi$ to arbitrary dimension is in terms of (ordinary) differential forms.  More precisely, define
\begin{equation}
 \label{eqn:general-rho}
 \begin{split}
  (\star p_\Phi(W))_{i_1\dotsm i_{2k}} & := \Phi_{s_1\dotsm s_k}^{t_1\dotsm t_k}W_{[i_1i_2|t_1|}{}^{s_1}\dotsm W_{i_{2k-1}i_{2k}]t_k}{}^{s_k}, \\
  (\Phi W^{k-1}C)_{i_2\dotsm i_{2k}} & := \Phi_{s_1\dotsm s_k}^{t_1\dotsm t_k}C_{t_1}{}^{s_1}{}_{[i_2}W_{|t_2|}{}^{s_2}{}_{i_3i_4} \dotsm W_{|t_k|}{}^{s_k}{}_{i_{2k-1}i_{2k}]}, \\
  (\star\rho^\Phi)_{i_2\dotsm i_{2k}} & := (\Phi W^{k-1}C)_{i_2\dotsm i_{2k}} - \frac{1}{n-4k}\nabla^i(\star p_\Phi(W))_{ii_2\dotsm i_{2k}} ,
 \end{split}
\end{equation}
where our notation in the first and second lines means that we skew symmetrize over the indices $i_1,\dotsc,i_{2k}$ and $i_2,\dotsc,i_{2k}$, respectively.  Note that these objects are defined without reference to a given orientation.  These normalizations are such that, in dimension $n=2k$, the definitions of $\rho_i^\Phi$ by \cref{eqn:rho-Phi} and the above display agree.  Moreover, $(\star\rho^\Phi)_{i_2\dotsm i_{2k}}$ is a conformally invariant $(2k-1)$-form of weight $-2$ in all dimensions; see \cref{sec:conf}.

This note is organized as follows.  In \cref{sec:bg} we recall some relevant facts from conformal geometry.  In \cref{sec:conf} we prove \cref{thm:invariant}.  In \cref{sec:proofs} we prove \cref{thm:image,thm:pfaffian}.  In \cref{sec:pontrjagin} we prove \cref{thm:pontrjagin-not-image}.  In \cref{sec:natural} we show that $\xi_i^{(2)}$ is not the divergence of a natural element of $\mE_{(ij)_0}[-2]$.

%% file: bg.tex
\section{Background}
\label{sec:bg}

\subsection{Abstract index notation}

Let $(M^n,g)$ be a pseudo-Riemannian manifold.  We denote by $T^{(r,s)}M$ the tensor product of the bundles $\otimes^r TM$ and $\otimes^sT^\ast M$.  We use abstract index notation~\cite{Penrose1984} to denote sections of tensor bundles.  Specifically, we denote a section of $T^{(r,s)}M$ by $T_{i_1\dotsm i_s}^{j_1\dotsm j_r}$; the $r$ distinct superscripts denote contravariant indices and the $s$ distinct subscripts denote covariant indices.  Repeated indices denote contractions between the corresponding components.  We use the metric $g_{ij}$ to raise and lower indices in the usual way, and often offset subscripts and superscripts to clarify which components are raised or lowered.  For example, as a section of $T^{(1,3)}M$, the Riemann curvature tensor is defined by
\[ R_{ij}{}^k{}_\ell X^\ell := \nabla_i\nabla_jX^k - \nabla_j\nabla_iX^k \]
for all vector fields $X^k$, where $\nabla_i$ is the Levi-Civita connection.  The Ricci curvature is $R_{ij}=R_{ki}{}^k{}_j$ and the scalar curvature is $R=R_k{}^k$.  The \emph{Schouten tensor} of $(M^n,g)$ is
\[ P_{ij} := \frac{1}{n-2}\left( R_{ij} - Jg_{ij} \right), \]
where $J=\frac{1}{2(n-1)}R$ is the trace of $P_{ij}$.  When clear from context, we write covariant derivatives of a scalar function using subscripts; e.g.\ given $f\in C^\infty(M)$, we may write $f_i$ for $\nabla_if$.

We use round and square brackets to denote symmetrization and skew symmetrization, respectively, over the enclosed indices.  For example, if $T_{ijk}$ is a section of $T^{(0,3)}M$, then
\begin{align*}
 T_{(ijk)} & := \frac{1}{3!}\left(T_{ijk} + T_{ikj} + T_{jki} + T_{jik} + T_{kij} + T_{kji}\right), \\
 T_{[ijk]} & := \frac{1}{3!}\left(T_{ijk} - T_{ikj} + T_{jki} - T_{jik} + T_{kij} - T_{kji}\right)
\end{align*}
denote the projections of $T_{ijk}$ to its symmetric and antisymmetric parts, respectively.  In this notation, the algebraic symmetries of the Weyl tensor $W_{ijkl}$ are expressed as
\[ W_{ijkl} = W_{[ij][kl]} = W_{[kl][ij]}, \quad W_{[ijk]l}=0, \quad W_{ikj}{}^k = 0, \]
which express that $W_{ijkl}$ is a section of $S^2\Lambda^2T^\ast M$, that $W_{ijkl}$ satisfies the first Bianchi identity, and that $W_{ijkl}$ is trace-free, respectively.  The differential symmetries of the Weyl tensor $W_{ijkl}$ and the Cotton tensor $C_{ijk}$ are also succinctly expressed in abstract index notation:

\begin{lemma}
 \label{lem:curvature_symmetries}
 Let $(M^n,g)$, $n\geq3$, be a Riemannian manifold.  Then
 \begin{align*}
  2\nabla_{[i}P_{j]k} & = C_{ijk} , \\
  \nabla_{[i}W_{jk]}{}^{lm} & = -2C_{[ij}{}^{[l}\delta_{k]}^{m]} .
 \end{align*}
\end{lemma}

\begin{proof}
 With our convention $\nabla^sW_{ijsk}=(n-3)C_{ijk}$ from the introduction, the first equation is the customary definition of the Cotton tensor.  The second equation follows from the second Bianchi identity $\nabla_{[i}R_{jk]lm}=0$.
\end{proof}

We use the symbol $\mE$ together with abstract indices to denote the spaces of sections of a given tensor bundle.  For example, $\mE^i$ denotes the space of sections of $TM$ and $\mE_{[i_1\dotsm i_k]}$ denotes the space of $k$-forms.  We denote by $\mE_{(ij)_0}$ the space of trace-free symmetric $(0,2)$-tensor fields.

Suppose for the moment that $(M^n,g)$ is oriented.  Denote by $\epsilon_{i_1\dotsm i_n}$ the pseudo-Riemannian volume form determined by $(M^n,g)$ and the orientation.  Given an integer $0\leq k\leq n$, the \emph{Hodge star operator} $\star\colon\mE_{[i_1\dotsm i_k]}\to\mE_{[i_{k+1}\dotsm i_n]}$ is defined by
\begin{equation}
 \label{eqn:Hodge-star}
 (\star\alpha)_{i_{k+1}\dotsm i_n} := \frac{1}{k!}\epsilon^{s_1\dotsm s_k}{}_{i_{k+1}\dotsm i_n}\alpha_{s_1\dotsm s_k} .
\end{equation}
A straightforward computation shows that
\[ \epsilon_{s_1\dotsm s_ki_{k+1}\dotsm i_n}\epsilon^{s_1\dotsm s_kj_{k+1}\dotsm j_n} = k!\delta_{i_{k+1}\dotsm i_n}^{j_{k+1}\dotsm j_n} . \]
This implies the familiar identity
\begin{equation*}
 (\star\star\alpha)_{i_1\dotsm i_k} = (-1)^{k(n-k)}\alpha_{i_1\dotsm i_k} .
\end{equation*}

\subsection{Conformal density bundles} Let $(M^n,c)$ be a conformal manifold (possibly of mixed signature).  The conformal class $c$ is naturally an $\bR_+$-principle bundle with $\bR_+$-action given by $s\cdot g_x=s^2g_x$ for all $s\in\bR_+$, all $g\in c$, and all $x\in M$.  Given $w\in\bR$, the \emph{conformal density bundle of weight $w$} is the line bundle associated to $c$ via the representation $s\mapsto s^{-w/2}\in\End(\bR)$ of $\bR_+$.  We denote by $\mE[w]$ the space of smooth sections of this bundle; equivalently, an element of $\mE[w]$ is an equivalence class of pairs $(f,g)\in C^\infty(M)\times c$ with respect to the equivalence relation $(f,g)\sim(e^{w\Upsilon}f,e^{2\Upsilon}g)$ for all $\Upsilon\in C^\infty(M)$.  Similarly, we denote by $\mE_{i_1\dotsm i_s}^{j_1\dotsm j_r}[w]$ the space of smooth sections of the tensor product of $T^{(r,s)}M$ with the conformal density bundle of weight $w$.

Recall that a tensor field $A_{i_1\dotsm i_s}^{j_1\dotsm j_r}$ is \emph{natural} if it can be written as a linear combination of partial contractions of the Riemannian metric, its inverse, the Riemann curvature tensor, and its covariant derivatives; when restricted to oriented manifolds, we also allow these products to include factors of the Riemannian volume form.  When $M$ is fixed, we may regard $A_{i_1\dotsm i_s}^{j_1\dotsm j_r}$ as a map from $\Met(M)$, the space of pseudo-Riemannian metrics on $M$, to $\mE_{i_1\dotsm i_s}^{j_1\dotsm j_r}$.

A \emph{natural element of $\mE_{i_1\dotsm i_s}^{j_1\dotsm j_r}[w]$} is an equivalence class $[ A_{i_1\dotsm i_s}^{j_1\dotsm j_r}(g),g]$, where $A_{i_1\dotsm i_s}^{j_1\dotsm j_r}$ is a natural tensor field.
We say that $[A_{i_1 \dotsm i_s}^{j_1 \dotsm j_r}(g),g]$ is \emph{conformally invariant} if it is independent of the choice of metric $g \in c$.
For example, $g_{ij}$ determines a natural conformally invariant element of $\mE_{(ij)}[2]$; $W_{ijkl}$ determines a natural conformally invariant element of $\mE_{ijkl}[2]$; and, if $(M,c)$ is oriented, then $\epsilon_{i_1\dotsm i_n}$ determines a natural conformally invariant element of $\mE_{[i_1\dotsm i_n]}[n]$.  In particular, we may use $g_{ij}$ to raise and lower indices in conformal density bundles, and hence, for example, identify $\mE^i[0]\cong\mE_i[2]$.

If $(M^n,c)$ is closed, then the total integral of any conformal density $f\in\mE[-n]$ is well-defined: simply pick $g\in c$, integrate against the Riemannian volume density of $g$, and observe that the result is independent of the choice of $g$.  It follows that there is a conformally invariant pairing $\mE_i[w]\times\mE_i[2-n-w]\to\bR$ given by
\begin{equation}
 \label{eqn:pairing}
 \lp \alpha_i, \beta_j\rp := \int_M g^{ij}\alpha_i\beta_j .
\end{equation}
These comments extend to general conformal manifolds by requiring $f$ or one of $\alpha_i,\beta_i$ to be compactly-supported.

The \emph{conformal Killing operator} $K\colon\mE_i[2]\to\mE_{(ij)_0}[2]$,
\[ K(\alpha_i) := 2\nabla_{(i}\alpha_{j)} - \frac{2}{n}\nabla^k\alpha_k g_{ij} , \]
is conformally invariant.  The kernel $\mK:=\ker K\subset\mE_i[2]$ of $K$ is (after raising the index) the space of conformal Killing fields.  The conformal invariance of \cref{eqn:pairing} and the analogous conformally invariant pairing of $\mE_{(ij)_0}[w]$ and $\mE_{(ij)_0}[4-n-w]$ implies that the formal adjoint $K^\ast\colon\mE_{(ij)_0}[2-n]\to\mE_i[-n]$,
\[ K^\ast(A_{ij}) := -2\nabla^kA_{ki}, \]
of $K$ is also conformally invariant.

\subsection{Infinitesimal conformal invariance}

Recall that a natural tensor field $T_{i_1\dotsm i_s}^{j_1\dotsm j_r}$ is \emph{homogeneous of degree $w\in\bR$} if
\[ T_{i_1\dotsm i_s}^{j_1\dotsm j_r}(c^2g) = c^w T_{i_1\dotsm i_s}^{j_1\dotsm j_r}(g) \]
for all $g\in\Met(M)$ and all constants $c>0$.  Given such a tensor field, conformal invariance is equivalent to infinitesimal conformal invariance~\cite{Branson1985}.  More precisely, given such a tensor field and a metric $g\in\Met(M)$, the \emph{conformal linearization of $T_{i_1\dotsm i_s}^{j_1\dotsm j_r}$ at $g$} is the map $D_gT_{i_1\dotsm i_s}^{j_1\dotsm j_r}\colon C^\infty(M)\to\mE_{i_1\dotsm i_s}^{j_1\dotsm j_r}$ defined by
\begin{equation}
 \label{eqn:conformal_linearization}
 D_gT_{i_1\dotsm i_s}^{j_1\dotsm j_r}(\Upsilon) := \left.\frac{\partial}{\partial t}\right|_{t=0} e^{-wt\Upsilon}T_{i_1\dotsm i_s}^{j_1\dotsm j_r}(e^{2t\Upsilon}g) .
\end{equation}
Observe that $D_gT_{i_1\dotsm i_s}^{j_1\dotsm j_r}$ is linear and annihilates constants.  One says that $T_{i_1\dotsm i_s}^{j_1\dotsm j_r}$ is \emph{infinitesimally conformally invariant} if $D_gT_{i_1\dotsm i_s}^{j_1\dotsm j_r}=0$ for all $g\in M$.  By integrating along paths in the conformal class $c$, one observes that $T_{i_1\dotsm i_s}^{j_1\dotsm j_r}$ is infinitesimally conformally invariant if and only if $T_{i_1\dotsm i_s}^{j_1\dotsm j_r}$ determines a natural conformally invariant element of $\mE_{i_1\dotsm i_s}^{j_1\dotsm j_r}[w]$.

Our proof of the conformal invariance of $\xi_i^{(k)}$ and $\rho_i^\Phi$ relies on three ingredients.  First are the well-known conformal linearizations of the Weyl and Cotton tensors.

\begin{lemma}
 \label{lem:curvature_transformations}
 Let $(M^n,g)$ be a pseudo-Riemannian manifold and let $\Upsilon\in C^\infty(M)$.  Then
 \begin{align*}
  D_gW_{ijkl}(\Upsilon) & = 0 , \\
  D_gC_{ijk}(\Upsilon) & = W_{ij}{}^s{}_k\Upsilon_s .
 \end{align*}
\end{lemma}

Second is the conformal linearization of the exterior derivative of a natural homogeneous scalar function.

\begin{lemma}
 \label{lem:derivative_transformation}
 Let $(M^n,g)$ be a pseudo-Riemannian manifold and let $\Upsilon\in C^\infty(M)$.  For any natural homogeneous Riemannian scalar function $f$ of degree $w$, it holds that
 \[ D_g\nabla_i f(\Upsilon) = wf\Upsilon_i + \nabla_i D_gf(\Upsilon) . \]
\end{lemma}

\begin{proof}
 This follows directly from \cref{eqn:conformal_linearization}.
\end{proof}

Third is the conformal linearization of the divergence of a natural homogeneous differential form.

\begin{lemma}
 \label{lem:connection_transformation}
 Let $(M^n,g)$ be a pseudo-Riemannian manifold and let $\Upsilon\in C^\infty(M)$.  For any natural homogeneous Riemannian $k$-form $\alpha_{i_1\dotsm i_k}$ of degree $w$, it holds that
 \[ D_g\nabla^i\alpha_{ii_2\dotsm i_k}(\Upsilon) = (n+w-2k)\Upsilon^i\alpha_{ii_2\dotsm i_k} + \nabla^i D_g\alpha_{ii_2\dotsm i_k}(\Upsilon) . \]
\end{lemma}

\begin{proof}
 This follows directly from \cref{eqn:conformal_linearization} and the fact that
 \[ \hnabla_i\alpha_j = \nabla_i\alpha_j -\Upsilon_i\alpha_j - \alpha_i\Upsilon_j + \Upsilon^s\alpha_s g_{ij} \]
 for all one-forms $\alpha_i$ and all metrics $g$ and $\hg=e^{2\Upsilon}g$ on $M$.
\end{proof}

%% file: conf.tex
\section{Conformal invariance}
\label{sec:conf}

In this section we prove \cref{thm:invariant}.  We separate the proof into two parts.

We begin by proving that $\xi_i^{(k)}$ is conformally invariant on $2k$-dimensional pseudo-Riemannian manifolds.

\begin{proposition}
 \label{prop:xi-invariant}
 Let $(M^{2k},g)$ be a pseudo-Riemannian manifold and define $\xi_i^{(k)}$ as in \cref{eqn:xi}.  For any $\Upsilon\in C^\infty(M)$, it holds that
 \[ e^{2k\Upsilon}\hxi_i^{(k)} = \xi_i^{(k)} , \]
 where $\hxi_i^{(k)}$ is defined in terms of $\hg:=e^{2\Upsilon}g$.
\end{proposition}

\begin{proof}
 As discussed in \cref{sec:bg}, it suffices to show that the conformal linearization of $\xi_i^{(k)}$ vanishes.  A direct computation using \cref{lem:curvature_transformations,lem:derivative_transformation} yields
 \begin{equation*}
  D\xi_i^{(k)}(\Upsilon) = \frac{1}{k!}\delta_{ii_2\dotsm i_{2k}}^{jj_2\dotsm j_{2k}} W_{jj_2}{}^{si_2} W_{j_3j_4}{}^{i_3i_4}\dotsm W_{j_{2k-1}j_{2k}}{}^{i_{2k-1}i_{2k}}\Upsilon_s - \frac{1}{2k}\Pf^{(k)}(W)\Upsilon_i .
 \end{equation*}
 Since $M$ is $2k$-dimensional, we conclude that
 \begin{align*}
  0 & = \frac{1}{k!}\delta_{ii_1\dotsm i_{2k}}^{jj_1\dotsm j_{2k}} W_{j_1j_2}{}^{i_1i_2}\dotsm W_{j_{2k-1}j_{2k}}{}^{i_{2k-1}i_{2k}}\Upsilon_j = -2kD\xi_i^{(k)}(\Upsilon) . \qedhere
 \end{align*}
\end{proof}

Let $\Phi$ be a homogeneous invariant polynomial of degree $k$.  We now turn to the proof that $\rho_i^\Phi$ is conformally invariant on oriented $2k$-dimensional pseudo-Riemannian manifolds.  We in fact prove the stronger claim that the $(2k-1)$-form $(\star\rho^\Phi)_{i_2\dotsm i_{2k}}$ defined by \cref{eqn:general-rho} is conformally invariant on any pseudo-Riemannian $n$-manifold.

\begin{proposition}
 \label{prop:general-rho-invariant}
 Let $\Phi$ be a homogeneous invariant polynomial of degree $k$, let $(M^n,g)$ be a pseudo-Riemannian manifold, and let $(\star\rho^\Phi)_{i_2\dotsm i_{2k}}$ be defined by \cref{eqn:general-rho}.  For any $\Upsilon\in C^\infty(M)$, it holds that
 \[ e^{2\Upsilon}(\widehat{\star\rho}^\Phi)_{i_2\dotsm i_{2k}} = (\star\rho^\Phi)_{i_2\dotsm i_{2k}}, \]
 where $(\widehat{\star\rho}^\Phi)_{i_2\dotsm i_{2k}}$ is defined in terms of $\hg:=e^{2\Upsilon}g$.
\end{proposition}

\begin{proof}
 As discussed in \cref{sec:bg}, it suffices to show that the conformal linearization of $(\star\rho^\Phi)_{i_2\dotsm i_{2k}}$ vanishes.  A direct computation using \cref{lem:curvature_transformations} yields
 \[ D_g(\Phi W^{k-1}C)_{i_2\dotsm i_{2k}}(\Upsilon) = \Phi_{s_1\dotsm s_k}^{t_1\dotsm t_k}W_{i[i_2|t_1|}{}^{s_1}W_{i_3i_4|t_2|}{}^{s_2}\dotsm W_{i_{2k-1}i_{2k}]t_k}{}^{s_k}\Upsilon^i . \]
 A direct computation using \cref{lem:connection_transformation} yields
 \[ D_g\nabla^i\bigl(\star p_\Phi(W)\bigr)_{ii_2\dotsm i_{2k}}(\Upsilon) = (n-4k)\Phi_{s_1\dotsm s_k}^{t_1\dotsm t_k}W_{i[i_2|t_1|}{}^{s_1}W_{i_3i_4|t_2|}{}^{s_2}\dotsm W_{i_{2k-1}i_{2k}]t_k}{}^{s_k}\Upsilon^i . \]
 Combining the previous two displays yields $D_g(\star\rho^\Phi)_{i_2\dotsm i_{2k}}=0$.
\end{proof}

\begin{corollary}
 \label{cor:rho-invariant}
 Let $\Phi$ be a homogeneous invariant polynomial of degree $k$, let $(M^{2k},g)$ be an oriented pseudo-Riemannian manifold, and define $\rho_i^\Phi$ as in \cref{eqn:rho-Phi}.  For any $\Upsilon\in C^\infty(M)$, it holds that
 \[ e^{2k\Upsilon}\hrho_i^\Phi = \rho_i^\Phi , \]
 where $\hrho_i^\Phi$ is defined in terms of $\hg:=e^{2\Upsilon}g$.
\end{corollary}

\begin{proof}
 It follows directly from \cref{eqn:Hodge-star} that $\rho_i^\Phi=-(\star\star\rho^\Phi)_i$.  The conclusion now follows from \cref{prop:general-rho-invariant} and the conformal invariance of the Hodge star operator $\star\colon\mE_{i_1\dotsm i_{2k-1}}[-2]\to\mE_i[-2k]$.
\end{proof}

Finally, combining \cref{prop:xi-invariant,cor:rho-invariant} yields \cref{thm:invariant}.

%% file: proofs.tex
\section{$\xi_i^{(k)}$ and the image of $K^\ast$}
\label{sec:proofs}

There are two steps in our proof that $\xi_i^{(k)}\in\im K^\ast$ on closed Riemannian $2k$-manifolds.  The first step is to write $\xi_i^{(k)}$ in a way that is manifestly orthogonal to the space of Killing fields.  We accomplish this by proving \cref{thm:pfaffian}.

\begin{proof}[Proof of \cref{thm:pfaffian}]
 First observe that
 \[ \tr\Omega^{(k)} = \sum_{\ell=0}^{k-1} \frac{4^{k-\ell}}{\ell!}\delta_{i_1\dotsm i_{k+\ell}}^{j_1\dotsm j_{k+\ell}} W_{j_1j_2}{}^{i_1i_2}\dotsm W_{j_{2\ell-1}j_{2\ell}}{}^{i_{2\ell-1}i_{2\ell}}P_{j_{2\ell+1}}^{i_{2\ell+1}}\dotsm P_{j_{k+\ell}}^{i_{k+\ell}} . \]
 Since $R_{ijkl} = W_{ijkl} + P_{ik}g_{jl} - P_{il}g_{jk} + P_{jl}g_{ik} - P_{jk}g_{il}$, we compute that
 \[ \Pf^{(k)}(\Rm) = \sum_{\ell=0}^k \frac{4^{k-\ell}}{\ell!}\delta_{i_1\dotsm i_{k+\ell}}^{j_1\dotsm j_{k+\ell}} W_{j_1j_2}{}^{i_1i_2}\dotsm W_{j_{2\ell-1}j_{2\ell}}{}^{i_{2\ell-1}i_{2\ell}}P_{j_{2\ell+1}}^{i_{2\ell+1}}\dotsm P_{j_{k+\ell}}^{i_{k+\ell}} . \]
 Combining these formulae yields
 \begin{equation}
  \label{eqn:trOmega}
  \tr\Omega^{(k)} = \Pf^{(k)}(\Rm) - \Pf^{(k)}(W) .
 \end{equation}

 Next, a straightforward computation using \cref{lem:curvature_symmetries} yields
 \begin{equation}
  \label{eqn:divOmega}
  \nabla^j\bigl(\Omega^{(k)}\bigr)_{ij} = \frac{2k}{k!}\delta_{ii_2\dotsm i_{2k}}^{jj_2\dotsm j_{2k}} C_{jj_2}{}^{i_2}W_{j_3j_4}{}^{i_3i_4}\dotsm W_{j_{2k-1}j_{2k}}{}^{i_{2k-1}i_{2k}} .
 \end{equation}
 The desired conclusion follows from \cref{eqn:trOmega,eqn:divOmega}.
\end{proof}

The second step is to apply the Ferrand--Obata Theorem.

\begin{proof}[Proof of \cref{thm:image}]
 Suppose first that $(M^{2k},g)$ admits an essential conformal Killing field $X$; i.e.\ $\mL_X\hg\not=0$ for all conformal metrics $\hg\in[g]$.  The Ferrand--Obata Theorem~\cites{Ferrand1996,Obata1971} implies that $g$ is locally conformally flat.  Hence $\xi_i^{(k)}=0$.
 
 Suppose instead that $(M^{2k},g)$ does not admit an essential conformal Killing field.  Let $X$ be a conformal Killing field.  Then there is a conformally equivalent metric $\hg\in[g]$ such that $\mL_X\hg=0$.  In particular, $\hnabla_i X^i=0$.  It follows from \cref{thm:invariant,thm:pfaffian} that
 \[ \int_M \xi_i^{(k)}X^i\,\dvol_g = \int_M \hxi_i^{(k)}X^i\,\dvol_{\hg} = -\frac{1}{4k^2}\int_M \Pf^{(k)}(\Rm_{\hg})\,\hnabla_i X^i\,\dvol_{\hg} = 0 . \]
 
 Now, since $(M,g)$ is Riemannian, the divergence $K^\ast\colon\mE_{(ij)_0}\to\mE_i$ has surjective principal symbol.  Therefore we have the $L^2$-orthogonal splitting
 \[ \mE_i = \im K^\ast \oplus \ker K . \]
 The previous two paragraphs imply that $\xi_i^{(k)}\in\im K^\ast$.  The final conclusion follows from conformal covariance.
\end{proof}

\begin{remark}
 \label{rk:other-signature}
 Our proof of \cref{thm:image} uses the fact that if $X^i\in\mK$ is essential, then $g$ is locally conformally flat~\cites{Ferrand1996,Obata1971}.  Frances~\cite{Frances2015} has constructed counterexamples to this statement for manifolds of signature $(p,q)$, $p,q\geq2$, though it remains unknown whether this statement holds in Lorentzian signature.  However, it is straightforward to check that $\xi_i^{(k)}=0\in\im K^\ast$ for Frances' even-dimensional counterexamples.  In particular, it is not known if \cref{thm:image} is false in non-Riemannian signatures.
\end{remark}

%% file: pontrjagin.tex
\section{$\rho_i^\Phi$ and the image of $K^\ast$}
\label{sec:pontrjagin}

The purpose of this section is to prove \cref{thm:pontrjagin-not-image}.  We separate the proof into two pieces, corresponding to the two conclusions of \cref{thm:pontrjagin-not-image}.

We first prove that the restriction of the induced functional $\Rho^\Phi$ to the space $\mK$ of conformal Killing fields vanishes on any closed conformal manifold of Riemannian signature which admits an Einstein metric.

\begin{proposition}
 \label{prop:einstein}
 Let $\Phi$ be a homogeneous invariant polynomial of degree $k\in\bN$ and let $(M^{2k},g)$ be a closed conformally Einstein manifold of Riemannian signature.  Then
 \[ \Rho^\Phi(X^i) := \int_M \rho_i^\Phi X^i\,\dvol = 0 \]
 for all conformal Killing fields $X^i\in\mK$, where $\rho_i^\Phi$ is defined by \cref{eqn:rho-Phi}.
\end{proposition}

\begin{proof}
 Since $\Rho^\Phi(X^i):=\int\rho_i^\Phi X^i\,\dvol$ is conformally invariant, we may assume that $(M^{2k},g)$ is Einstein.  Hence
 \[ \rho_i^\Phi = \frac{1}{2k}\nabla_i p_\Phi(W) . \]
 Let $X^i\in\mK$.  Obata~\cite{Obata1962} proved that either $X^i$ is Killing or $(M^{2k},g)$ is isometric to the round $2k$-sphere.  In the former case,
 \[ \Rho^\Phi(X^i) = -\frac{1}{2k}\int_M p_\Phi(W)\nabla^iX_i\,\dvol = 0 . \]
 In the latter case, $\rho_i^\Phi=0$, and hence $\Rho^\Phi(X^i)=0$.
\end{proof}

We now construct examples of closed Riemannian $4k$-manifolds and homogeneous invariant polynomials $\Phi$ of degree $2k$ for which $\Rho^\Phi\rv_\mK\not=0$.  To that end, let $\bH$ denote the space of quaternions and let $X,Y,Z$ be the frame of left-invariant vector fields on $S^3\subset\bR^4\cong\bH$ which restrict to $i,j,k$ at the identity.  Let $\alpha,\beta,\gamma$ be the dual coframe.  Given $t>0$, the \emph{Berger sphere} is the Riemannian manifold $(S^3,g_t)$, where
\[ g_t := t \alpha\otimes\alpha + \beta\otimes\beta + \gamma\otimes\gamma . \]

We begin by finding an example in dimension four.

\begin{proposition}
 \label{prop:berger}
 Fix $\Phi_{ij}^{rs}=\frac{1}{2}\delta_i^s\delta_j^r$.  Let $(S^3,g_t)$, $t>0$, be a Berger sphere and let $\theta$ be a nonvanishing left-invariant one-form on $S^1$.  If $t\not=1$, then the Riemannian product $(S^3\times S^1,\og_t:=g_t+ \theta^2)$ is such that
 \[ \Rho^\Phi\rv_{\mK} \not= 0 . \]
 In particular, if $t\not=1$, then $(S^3\times S^1,\og_t)$ is not conformal to an Einstein metric.
\end{proposition}

%\begin{remark}
% The last conclusion is not a surprise: the rigidity~\cite{Hitchin1974} of the Hitchin--Thorpe inequality implies that there is no Einstein metric on  $S^3\times S^1$.
%\end{remark}

\begin{remark}
 If $t\not=1$, then $(S^3 \times S^1 , \og_t)$ is not Bach flat, and hence not even locally conformally Einstein.
 We are not aware of an example of a closed, locally conformally Einstein four-manifold which can be shown to not be conformally Einstein using \cref{prop:einstein}.
\end{remark}

\begin{proof}
 For clarity of the exposition, we write this proof in index-free notation.
 
 It is well-known that
 \begin{align*}
  \nabla^{g_t}\alpha & = -\beta\otimes\gamma + \gamma\otimes\beta, \\
  \nabla^{g_t}\beta & = -(t-2)\alpha\otimes\gamma - t\gamma\otimes\alpha, \\
  \nabla^{g_t}\gamma & = (t-2)\alpha\otimes\beta + t\beta\otimes\alpha, \\
  \Ric_{g_t} & = 2t^2\alpha\otimes\alpha + 2(2-t)\beta\otimes\beta + 2(2-t)\gamma\otimes\gamma .
 \end{align*}
 From this it readily follows that
 \begin{align*}
  W^{\og_t} & = \frac{2(t-1)}{3}\Bigl[ t(\alpha\wedge\beta)\otimes(\alpha\wedge\beta) + t(\alpha\wedge\gamma)\otimes(\alpha\wedge\gamma) - 2(\beta\wedge\gamma)\otimes(\beta\wedge\gamma) \\
   & \qquad - 2t(\alpha\wedge\theta)\otimes(\alpha\wedge\theta) + (\beta\wedge\theta)\otimes(\beta\wedge\theta) + (\gamma\wedge\theta)\otimes(\gamma\wedge\theta) \Bigr], \\
   C^{\og_t} & = 2t(t-1)\left[ (\alpha\wedge\beta)\otimes\gamma - (\alpha\wedge\gamma)\otimes\beta - 2(\beta\wedge\gamma)\otimes\alpha\right] .
 \end{align*}
 We deduce that $p_\Phi(W)=0$ and
 \[ \star\rho^\Phi = \Phi WC = -\frac{8t(t-1)^2}{3}\alpha\wedge\beta\wedge\gamma . \]
 
 Let $T$ be the vector field on $S^1$ dual to $\theta$.  Then $T$ is a Killing field for $(S^3\times S^1,\og_t)$.  We compute that
 \[ \Rho^\Phi(T) = \frac{8t(t-1)^2}{3}\int_{S^3\times S^1} \alpha \wedge \beta \wedge \gamma \wedge \theta . \]
 In particular, if $t\not=1$, then $\Rho^\Phi\rv_{\mK}\not=0$.  The final conclusion follows from \cref{prop:einstein}.
\end{proof}

Taking Riemannian products with $k-1$ copies of $\bCP^2$ yields examples in general dimension $4k$.

\begin{proposition}
 \label{prop:berger_product}
 Let $\Phi$ be the homogeneous invariant polynomial of degree $2k$, $k\in\bN$, such that
 \[ \Phi_{s_1\dotsm s_{2k}}^{t_1\dotsm t_{2k}}\omega_{t_1}{}^{s_1}\dotsm\omega_{t_{2k}}{}^{s_{2k}} = \bigl(\omega_{rs}\omega^{sr}\bigr)^k \]
 for all $\omega_{ij}\in\mE_{ij}$.  Let $t>0$ and consider the Riemannian product
 \[ \bigl(S^3\times S^1 \times \underbrace{\bCP^2 \times \dotsm \times \bCP^2}_{\text{$k-1$ times}}, G_t := \og_t + \underbrace{g_{FS} + \dotsm + g_{FS}}_{\text{$k-1$ times}} \bigr) \]
 of $(S^3\times S^1,\og_t)$ with $k-1$ copies of $\bCP^2$ equipped with the Fubini--Study metric $g_{FS}$.  If $t\not=1$, then $\Rho^\Phi\rv_{\mK} \not=0$.
\end{proposition}

\begin{proof}
 Let $\cPhi$ be the invariant polynomial of \cref{prop:berger}.
 
 First observe that $p_\Phi(W_{G_t})$ is a nonzero multiple of $p_1(\bCP^2)^{k-1}\wedge p_{\cPhi}(W_{\og_t})$.  As noted in the proof of \cref{prop:berger}, it holds that $p_{\cPhi}(W_{\og_t})=0$.  Therefore $p_\Phi(W_{G_t})=0$.
 
 Next observe that $\Phi W_{G_t}^{2k-1}C_{G_t}$ is a nonzero multiple of $p_1(\bCP^2)^{k-1}\wedge\star\rho_{\og_t}^{\cPhi}$.  Since $\int_{\bCP^2} p_1(\bCP^2)\not=0$, we conclude that $\Rho_{G_t}^\Phi(T)$ is a nonzero multiple of $\Rho_{\og_t}^{\cPhi}(T)$.  Hence, by the proof of \cref{prop:berger}, it holds that $\Rho_{G_t}^\Phi\rv_{\mK}\not=0$.
\end{proof}

Finally, combining \cref{prop:berger,prop:berger_product,prop:einstein} yields \cref{thm:pontrjagin-not-image}.

%% file: natural.tex
\section{$\xi_i^{(2)}$ and the divergence of natural tensors}
\label{sec:natural}

We conclude by proving that, in dimension four, the natural conformal invariant
\begin{equation*}
 \xi_i^{(2)} = 2W_{istu}C^{tus} + \frac{1}{8}\nabla_i (W_{stuv}W^{stuv})
\end{equation*}
is not expressible as the divergence of a natural symmetric $(0,2)$-tensor field of weight $-2$.  This follows from the classification of the natural elements of $\mE_{(ij)_0}[-2]$ in dimension four.

\begin{proposition}
 \label{prop:unnatural}
 In dimension four, the vector space of natural elements of $\mE_{(ij)_0}[-2]$ is spanned by the set
 \begin{equation}
  \label{eqn:spanning-set}
  \left\{ B_{ij}, W_{isjt}P^{st}, \tf P_i^sP_{sj}, \tf JP_{ij}, \tf\nabla_{ij}^2J \right\} .
 \end{equation}
 In particular, $\xi_i^{(2)}$ is not the divergence of a natural element of $\mE_{(ij)_0}[-2]$.
\end{proposition}

\begin{proof}
 On a pseudo-Riemannian four-manifold, the space of natural symmetric $(0,2)$-tensor fields of weight $-2$ is spanned by partial contractions of $\nabla^2\Rm\otimes g$ and $\Rm\otimes\Rm\otimes g$.  Equivalently, it is spanned by $\Delta P$, $\nabla_{ij}^2J$, ${\check W}_{ij}^2:=W_{istu}W_j{}^{stu}$, $W_{isjt}P^{st}$, $P_i^sP_{sj}$, $JP_{ij}$, and products of their traces with $g_{ij}$.  Using the facts that, in dimension four,
 \[ B_{ij} = \Delta P_{ij} - \nabla_{ij}^2 J + 2W_{isjt}P^{st} - 4P_i^sP_{sj} + \lv P\rv^2g_{ij} \]
 and $\tf{\check W}_{ij}^2=0$, we conclude that the space of natural elements of $\mE_{(ij)_0}[-2]$ is spanned by \cref{eqn:spanning-set}.
 
 Next, it is known that the Bach tensor is divergence-free~\cite{FeffermanGraham2012}.
 Direct calculation gives
 \begin{align*}
  \nabla^j(W_{isjt}P^{st}) & = -C_{sit}P^{st} + \frac{1}{2}W_{istu}C^{tus} , \\
  \nabla^j(\tf P_i^sP_{sj}) & = P_{is}\nabla^sJ + \frac{1}{4}\nabla_i(P_{st}P^{st}) + C_{sit}P^{st} , \\
  \nabla^j(\tf JP_{ij}) & = P_{is}\nabla^sJ + \frac{1}{4}\nabla_i(J^2) , \\
  \nabla^j(\tf\nabla_{ij}^2J) & = \frac{3}{4}\nabla_i\Delta J + 2P_{is}\nabla^sJ + \frac{1}{2}\nabla_i(J^2) . \\
 \end{align*}
 It readily follows that there is not a natural element of $\mE_{(ij)_0}[-2]$ with divergence equal to $\xi_i^{(2)}$.
%  trace-free symmetric $(0,2)$-tensor fields of weight $-2$ on four-dimensional manifolds is spanned by
% \[ \left\{ B_{ij}, W_{isjt}P^{st}, \tf P_i^sP_{sj}, \tf JP_{ij}, \tf\nabla_{ij}^2J \right\} . \]
% It is known that $B$ is conformally invariant~\cite{FeffermanGraham2012}, while straightforward computations yield
% \begin{align*}
%  D_g W_{isjt}P^{st}(\Upsilon) & = -W_{isjt}\Upsilon^{st} , \\
%  D_g\tf P_i^sP_{sj}(\Upsilon) & = -2P_i^s\Upsilon_{sj} + \frac{1}{2}\lp P,\nabla^2\Upsilon\rp g_{ij} , \\
%  D_g\tf JP_{ij}(\Upsilon) & = -J\Upsilon_{ij} - (\Delta\Upsilon)P_{ij} + \frac{1}{2}J\Delta\Upsilon g_{ij}, \\
%  D_g\tf\nabla_{ij}^2J(\Upsilon) & = -\nabla_{ij}^2\Delta\Upsilon - 2J\nabla_{ij}^2\Upsilon - 3(\Upsilon_i\nabla_j J + \Upsilon_j\nabla_i J) \\
%   & \quad + \frac{1}{4}\left(\Delta^2\Upsilon + 2J\Delta\Upsilon + 6\lp\nabla\Upsilon,\nabla J\rp \right)g .
% \end{align*}
% Suppose that $I_{a,b,c,e}:=aW_{isjt}P^{st}+b\tf P_i^sP_{sj}+c\tf JP_{ij}+e\tf\nabla_{ij}^2J$ is conformally invariant.  Then $D_gI_{a,b,c,e}(\Upsilon)=0$ for all Riemannian four-manifolds $(M^4,g)$ and all $\Upsilon\in C^\infty(M)$.  It is straightforward to exploit this significant freedom to show that $a=b=c=e=0$ (cf.\ \cite{Graham1992}).  This yields the first claim.
% 
% Finally, since the Bach tensor is divergence-free~\cite{FeffermanGraham2012}, we conclude that the divergence of any natural element of $\mE_{(ij)_0}[-2]$ is zero.  However, $\xi_i^{(2)}$ is nontrivial; e.g.\ it is nonzero on many K3 surfaces~\cite{CaseTakeuchi2019}.
\end{proof}

It is natural to conjecture that $\xi_i^{(k)}$, $k\geq2$, cannot be expressed as the divergence of a natural element of $\mE_{(ij)}[2-2k]$ in dimension $2k$.
However, an attempt to verify this by identifying a basis for $\mE_{(ij)_0}[2-2k]$ is impractical for general $k$.